\documentclass[11pt, a4paper]{amsart}

\usepackage{amsthm}
\usepackage[T1]{fontenc}
\usepackage[utf8]{inputenc}
\usepackage{lmodern}
\usepackage[american]{babel}
\usepackage[babel, final]{microtype}
\usepackage{amssymb}
\usepackage{mathtools}
\usepackage{ifpdf}
\usepackage{comment}
\usepackage{multirow}
\usepackage{dutchcal} 
\usepackage{faktor} 
\usepackage[shortlabels]{enumitem}

\usepackage[pdftex, dvipsnames]{xcolor}
\definecolor{darkblue}{rgb}{0,0,0.6}
\usepackage[pdftex,final]{graphicx} 
\usepackage{caption}
\usepackage{subcaption}
\usepackage{pinlabel}
\usepackage[pdftex,%
    a4paper,%
    includehead,%
    includefoot,%
    nomarginpar,%
    lmargin=1in,%
    rmargin=1in,%
    tmargin=1in,%
    bmargin=1in,%
]{geometry}
\usepackage[breaklinks,%
    pdftex,%
    final,%
    colorlinks=true,%
    linkcolor=NavyBlue,%
    citecolor=blue,
    filecolor=NavyBlue,%
    menucolor=NavyBlue,%
    urlcolor=black,%
    bookmarks=true,%
    bookmarksdepth=3,%
    bookmarksnumbered=true,%
    bookmarksopen=true,%
    bookmarksopenlevel=2,%
    unicode
]{hyperref}
\usepackage[capitalize]{cleveref}
\usepackage{overpic}
\usepackage{tikz,tikz-cd}
\usetikzlibrary{tikzmark}

\usepackage{titlesec}
\usepackage{titletoc}
\usepackage{imakeidx}
\makeindex

\usepackage{adjustbox}

\makeatletter
\g@addto@macro\@floatboxreset\centering
\makeatother

\allowdisplaybreaks[4]

\addto\extrasamerican{%
}

\theoremstyle{definition}

\theoremstyle{remark}

\makeatletter
\let\c@conjecture=\c@theorem
\let\c@corollary=\c@theorem
\let\c@observation=\c@theorem
\let\c@proposition=\c@theorem
\let\c@lemma=\c@theorem
\let\c@definition=\c@theorem
\let\c@example=\c@theorem
\let\c@remark=\c@theorem
\let\c@question=\c@theorem
\let\c@equation\c@theorem
\makeatother

\def\makeautorefname#1#2{\expandafter\def\csname#1autorefname\endcsname{#2}}
\makeautorefname{theorem}{Theorem}%
\makeautorefname{conjecture}{Conjecture}%
\makeautorefname{corollary}{Corollary}%
\makeautorefname{observation}{Observation}%
\makeautorefname{proposition}{Proposition}%
\makeautorefname{lemma}{Lemma}%
\makeautorefname{definition}{Definition}%
\makeautorefname{example}{Example}%
\makeautorefname{remark}{Remark}%
\makeautorefname{question}{Question}%

\numberwithin{equation}{section}

\newcommand*{\Alphabet}{ABCDEFGHIJKLMNOPQRSTUVWXYZ1234567890}
\newcommand*{\alphabet}{abcdefghijklmnopqrstuvwxyz1234567890}
\newlength\fcaph
\newlength\fdesc
\newlength\factualfontsize
\numberwithin{equation}{section}
\newcounter{commentcounter}

\usepackage{letltxmacro}

\LetLtxMacro\Oldfootnote\footnote

 \usepackage{needspace}

\titleformat{\section}
  {\large\bfseries\centering}
  {\thesection}{1em}{}

\titleformat{\subsection}
  {\normalsize\bfseries\centering}
  {\thesubsection}{1em}{}

\titleformat{name=\subsection,numberless}
  {\normalsize\bfseries\centering}
  {}{1em}{}

\titleformat{\subsubsection}[runin]
  {\normalfont\normalsize\bfseries}
  {\thesubsubsection}{5pt}{}[.]

\titleformat{name=\subsubsection,numberless}[runin]
  {\normalfont\normalsize\bfseries}
  {}{0pt}{}

\titlespacing{\subsection}{0pt}{*2.5}{*1.5}

\titlecontents{section}[1.5em]
    {}
    {\normalfont\small\contentslabel{1.5em}}
    {\normalfont\small}
    {\dotfill \contentspage}{}
\titlecontents{subsection}[5em]
    {}
    {\normalfont\small\contentslabel{2.5em}}
    {}
    {\hfill \contentspage}{}
    
\makeatletter
\newcommand*\bigcdot{\mathpalette\bigcdot@{.7}}
\newcommand*\bigcdot@[2]{\mathbin{\vcenter{\hbox{\scalebox{#2}{$\m@th#1\bullet$}}}}}
\makeatother

\makeatletter
\DeclareFontFamily{U}{mathx}{\hyphenchar\font45}
\DeclareFontShape{U}{mathx}{m}{n}{
      <5> <6> <7> <8> <9> <10>
      <10.95> <12> <14.4> <17.28> <20.74> <24.88>
      mathx10
      }{}
\DeclareSymbolFont{mathx}{U}{mathx}{m}{n}
\DeclareFontSubstitution{U}{mathx}{m}{n}

\let\widecheck\@undefined
\let\widebar\@undefined
\DeclareMathAccent{\widecheck}{\mathord}{mathx}{"71}
\DeclareMathAccent{\widebar}{\mathord}{mathx}{"73}
\makeatother

\makeatletter
\DeclareSymbolFont{stmry}{U}{stmry}{m}{n}
\SetSymbolFont{stmry}{bold}{U}{stmry}{b}{n}

\let\llbracket\@undefined
\let\rrbracket\@undefined
\DeclareMathDelimiter{\llbracket}{\mathopen}%
                     {stmry}{"4A}{stmry}{"71}
\DeclareMathDelimiter{\rrbracket}{\mathclose}%
                     {stmry}{"4B}{stmry}{"79}
\makeatother

\makeatletter
\newcommand{\colim@}[2]{%
  \vtop{\m@th\ialign{##\cr
    \hfil$#1\operator@font colim$\hfil\cr
    \noalign{\nointerlineskip\kern1.5\ex@}#2\cr
    \noalign{\nointerlineskip\kern-\ex@}\cr}}%
}
\newcommand{\colim}{%
  \mathop{\mathpalette\colim@{\rightarrowfill@\textstyle}}\nmlimits@
}
\makeatother

\tikzset{
math to/.tip={Glyph[glyph math command=rightarrow]},
loop/.tip={Glyph[glyph math command=looparrowleft, swap]},
weird/.tip={Glyph[glyph math command=Rrightarrow, glyph length=1.5ex]},
pi/.tip={Glyph[glyph math command=pi, glyph length=1.5ex, glyph axis=0pt]},
}

\newcommand{\topint}{\mathring} 

\newcommand\rsmraise[1]{%
  \ifx#1\displaystyle .5\else
    \ifx#1\textstyle .5\else
      \ifx#1\scriptstyle .3\else
        .45%
      \fi
    \fi
  \fi}

\theoremstyle{plain}
\newtheorem{thm}{Theorem}[section]

\newtheorem{coro}[thm]{Corollary}
\newtheorem{lem}[thm]{Lemma}
\newtheorem{prop}[thm]{Proposition}

\theoremstyle{definition}
\newtheorem{define}[thm]{Definition}
\newtheorem{exa}[thm]{Example}

\theoremstyle{remark}
\newtheorem*{cla}{Claim}
\newtheorem*{rem}{Remark}

\title{
Topological manifolds are locally flat Euclidean Neighbourhood Retracts
}

\author{Raphael Floris}
\address{Universit\"{a}t Bonn}
\email{s6raflor@uni-bonn.de}

\begin{document}

\begin{abstract}
We show that every topological $n$-manifold $M$ admits a locally flat closed embedding $\iota\colon M \hookrightarrow \mathbb{R}^{2n+1}$ and is a retract of some neighbourhood $U \subseteq \mathbb{R}^{2n+1}$.
\end{abstract}

\maketitle

\section{Introduction}
\vspace*{0.5cm}
The 19\textsuperscript{th} century conception of a \textit{manifold}, a term first coined by Bernhard Riemann and used extensively in his seminal work \cite{Rie19}, was a turning point for the fields of study that can best be subsumed under the term \textit{geometry}. \\ Two viewpoints on manifolds consolidated at the beginning of the 20\textsuperscript{th} century. An \textit{extrinsic manifold} $M$ resides inside an ambient Euclidean $N$-space (in modern terminology, we would say that $M$ is a \textit{submanifold} of $\mathbb{R}^N$). The \textit{intrinsic} viewpoint, i.e. the definition of a manifold via charts, was first formalized by Hermann Weyl in \cite{Wey13}. It was Hassler Whitney, who showed in \cite{Whi36} that the extrinsic and the intrinsic point of view are equivalent in the case of smooth manifolds. 
\begin{thm}[Weak Whitney Embedding Theorem]

Every smooth $n$-manifold (with or without boundary) $M$ admits a closed smooth embedding $\iota \colon M \hookrightarrow \mathbb{R}^{2n+1}$.
\end{thm}


Something more is true. Consider a closed smooth submanifold $M \subseteq \mathbb{R}^N$. Then $M$ possesses a certain \textquotedblleft nice\textquotedblright {} neighbourhood: There is a neighbourhood $U \subseteq \mathbb{R}^N$ of $M$ such that $M$ is a \textit{retract} of $U$, i.e. there exists a continuous map $r\colon U \to M$ such that $r|_M = \mathrm{id}_M$. For the proof, one considers the normal bundle $\pi\colon NM \to M$ of $M$ and constructs an open subset $V \subseteq NM$ such that the map $\theta\colon V \to \mathbb{R}^N, (x,v) \mapsto x + v$ is a diffeomorphism onto its image. The neighbourhood $U = \theta(V)$ is called a \textit{tubular neighbourhood} of $M$. Thus every smooth manifold is a \textit{Euclidean Neighbourhood Retract (ENR)} (see Definition 4.1). For a more detailed account see \cite[Chapter 6]{Lee13}. \vspace*{0.25cm}






The aim of this paper is to present a readable exposition of the proof of the corresponding results for topological manifolds. We want to show that every topological $n$-manifold can be realised as a closed submanifold of $\mathbb{R}^{2n+1}$ that is the retract of some neighbourhood. More concretely:
\begin{thm}
Every topological $n$-manifold $M$ admits a locally flat closed embedding~$\iota\colon M \to \mathbb{R}^{2n+1}$. Furthermore $M$ admits a normal microbundle $\mathfrak{n}_{\iota\colon M \to \mathbb{R}^{2n+1}}$ and is thus an ENR.
\end{thm}
The fact that every compact topological $n$-manifold $M$ admits a locally flat closed embedding $\iota\colon M \hookrightarrow \mathbb{R}^d$ for some positive integer $d$ is well-known. \cite{KK} constructs a closed embedding, not necessarily locally flat, embedding $\iota\colon M \hookrightarrow \mathbb{R}^{(n+1)(n+2)+1}$ for an arbitrary manifold $M$. This paper is the first to construct a locally flat embedding $\iota\colon M \hookrightarrow \mathbb{R}^{2n+1}$, thus proving that $d = 2n+1$ suffices as in the smooth case. Whitney showed in \cite{Whi43} that in the smooth case we can even choose $d = 2n$. It is an interesting question whether Theorem 1.2 holds with $2n$ instead of $2n+1$. \vspace*{0.25cm}

We need \textit{local flatness} to ensure that the image of the embedding is a submanifold of $\mathbb{R}^{2n+1}$. 
\begin{define}
Let $N$ be a topological $n$-manifold without boundary. A subset $S \subseteq N$ is a \emph{locally flat submanifold of} $N$ if there exists some $0 \leq k \leq n$ such that for each $p \in S$, there exists a chart $(U,\varphi) = (U,x^1,\ldots,x^n)$ of $N$ such that $\varphi(p) = 0$ and either
\begin{align*}U \cap S = \{q \in U \; | \; x^{k+1}(q) = \cdots = x^n(q) = 0 \} \end{align*} or 
\begin{align*} U \cap S = \{q \in U \; | \; x^{k}(q) \geq 0 \; \mathrm{and} \;  x^{k+1}(q) = \cdots = x^n(q) = 0 \}. \end{align*}
If $M$ is a topological $m$-manifold (with or without boundary) and $\iota\colon M \hookrightarrow N$ is a topological embedding, we call $\iota$ \emph{locally flat} if $\iota(M)$ is a locally flat submanifold of $N$.
\end{define}

The second part of Theorem 1.2 is essentially due to \cite{Ste75} as we will explain in Section 4. Another way to show that every topological manifold is an ENR uses that every topological manifold is an \textit{Absolute Neighbourhood Retract (ANR)} and is due to \cite{Han51}. We will also explain this in Section 4. Showing the first sentence of Theorem 1.2 is divided into two steps: First we show that $M$ admits a topological proper embedding $\iota_0\colon M \hookrightarrow \mathbb{R}^{2n+1}$ that is possibly not locally flat. The existence of a locally flat closed embedding $\iota\colon M \hookrightarrow \mathbb{R}^{2n+1}$ for $n \geq 2$ then follows essentially from the following result applied to $N = \mathbb{R}^{2n+1}$.
\begin{thm}[{\cite[Theorem 5]{Mil72}}]
Let $M$ be a topological $m$-manifold, let $N$ be a topological $n$-manifold without boundary and let $\varepsilon > 0$. Choose a metric $d$ on $N$ that induces the topology of~$N$. If $\iota_0 \colon M \hookrightarrow N$ is a proper topological embedding and $n - m \leq 3 $, then there exists a proper locally flat embedding $\iota\colon M \hookrightarrow N$ such that $d(\iota_0(p),\iota(p)) < \varepsilon$ for all $p \in M$.
\end{thm}
If a continuous map $f\colon X \to Y$ between locally compact Hausdorff spaces $X$ and $Y$ is proper, then $f$ is also closed (\cite[Chapter I, Proposition 11.5]{Bre97}). Thus, the proper embedding $\iota\colon M \hookrightarrow \mathbb{R}^{2n+1}$ is also closed. \\
For $n = 1$ the existence of a locally flat embedding follows from the classification of $1$-manifolds. If $M$ is a connected $1$-manifold, then $M$ is homeomorphic to $X$, where $X \in \{[0,1],\mathbb{S}^1,[0,\infty),\mathbb{R}\}$, all of which admit a locally flat embedding into $\mathbb{R}^2 \subseteq \mathbb{R}^3$. \vspace*{0.25cm} 

Showing the existence of a topological embedding makes use of topological dimension theory, whose fundamentals we recapitulate in the next section.

\begin{rem}
Throughout this work, we denote by $\mathbb{N} := \{1,2,\ldots\}$ the set of positive integers and by $\mathbb{N}_0 := \{0,1,2,\ldots\}$ the set of non-negative integers. A \emph{manifold} can be with or without boundary. 
\end{rem}

\subsection*{Acknowledgements} 
This paper grew out of a graduate student seminar on topological manifolds at the University of Bonn. I would like to thank the organisers, Mark Powell and Arunima Ray, for suggesting this topic. I thank Mark Powell for encouraging me to turn the notes of my seminar talk into a paper. His comments were indispensable for this task. Furthermore, I would like to thank Stefan Friedl for pointing out some mistakes in my original seminar notes.

\subsection*{Organization of the paper} 
The main part of the paper is divided into three parts. In the first part, we recapitulate the basics of Dimension Theory and show that $\mathrm{dim}~C \leq n$ for every compact subset $C$ of an $n$-manifold $M$. This result is needed to prove the existence of a proper topological embedding $\iota\colon M \hookrightarrow \mathbb{R}^{2n+1}$ in the next part. The proof presented in this part is based on \cite[p. 315, Exercise 6]{Mun00}. The last part is dedicated to proving that $M$ is an ENR and admits a normal microbundle $\mathfrak{n}$.
\section{Dimension Theory}

\begin{define}Let $X$ be a topological space and let $\mathcal{U}$ be an open covering of $X$. A \emph{refinement} of $\mathcal{U}$ is an open cover $\mathcal{V}$ of $X$ such that every $V \in \mathcal{V}$ is contained in some $U \in \mathcal{U}$, i.e. $V \subseteq U$. \end{define}

\begin{define}Let $X$ be a topological space. \begin{enumerate}
  \item A collection $\mathcal{A}$ of subsets of $X$ has \emph{order} $m \in \mathbb{N}_0$ if $m$ is the largest integer such that there are $m + 1$ elements of $\mathcal{A}$ having a non-empty intersection. 
  \item $X$ is called \emph{finite-dimensional} if there exists some $m \in \mathbb{N}_0$ such that every open cover of $X$ possesses a refinement of order at most $m$. \\ The smallest such $m$ is called the \emph{(topological) dimension} of $X$, denoted by $\mathrm{dim}~X$. 
\end{enumerate}
\end{define} 

If $X$ is a topological space and $\mathcal{A}$ is a collection of subsets of $X$, then $\mathcal{A}$ has order $m$ if and only if there exists some $x \in X$ that lies in $m+1$ elements of $\mathcal{A}$ and no point of $X$ lies in more than $m+1$ elements of $\mathcal{A}$. 
\vspace*{0.5cm} 

Let us illuminate the notion of topological dimension with an example.

\begin{exa} 
Let $I := [0,1]$ denote the closed unit interval. We want to show that $\mathrm{dim}~I = 1$. 
Let $\mathcal{U}$ be an open cover of $I$. Since $I$ is a compact metric space, $\mathcal{U}$ has a positive Lebesgue number $\lambda > 0$, i.e. every subset of $I$ having diameter less than $\lambda$ is contained in an element of $\mathcal{U}$. For $k \in \mathbb{N}_0$, let $J_k := \left(\left(k-1\right) \cdot \frac{\lambda}{4}, \left(k+1\right) \cdot \frac{\lambda}{4}\right)$. Since $\mathrm{diam}~J_k = \frac{\lambda}{2} < \lambda$, we can conclude that $\mathcal{V} := \{J_k \cap I \}_{k \in \mathbb{N}_0}$ is a refinement of $\mathcal{U}$. Since $\mathcal{V}$ has order $1$, this shows $\mathrm{dim}~I \leq 1$. \\ In order to show that $\mathrm{dim}~I \geq 1$, we consider the open over $\mathcal{U} := \{[0,1), (0,1]\}$. If $\mathrm{dim}~I = 0$, $\mathcal{U}$ would have a refinement $\mathcal{V}$ of order $0$. Since $\mathcal{V}$ refines $\mathcal{U}$, we get $\mathrm{card}\left(\mathcal{V}\right) \geq 2$ (note that $0 \in V_1$ and $1 \in V_2$ for some $V_1,V_2 \in \mathcal{V}$ and because $\mathcal{V}$ refines $\mathcal{U}$, we get $V_1 \subseteq [0,1)$ and $V_2 \subseteq (0,1]$ and thus $V_1 \neq V_2$). Let $V$ be any element of $\mathcal{V}$ and let $W$ be the union of all $V' \in \mathcal{V} \setminus {V}$. Then both $V$ and $W$ are open and $V \cup W = I$ and $V \cap W = \emptyset$, because $\mathcal{V}$ has order $0$, which is a contradiction since $I$ is connected. \\ Thus, $\mathrm{dim}~I \geq 1$ and therefore $\mathrm{dim}~I = 1$.
\end{exa}

We can use Lebesgue numbers to show a more general result that will be needed throughout this section.

\begin{prop}
Let $n \in \mathbb{N}$. Every compact subspace of $\mathbb{R}^n$ has topological dimension at most~$n$.\end{prop}
\begin{proof}
Let us first divide $\mathbb{R}^n$ into unit cubes. Let \begin{align*}\mathcal{J} &:= \{(k,k+1)\}_{k \in \mathbb{Z}} \\ \mathcal{K} &:= \{\{k\}\}_{k \in \mathbb{Z}}. \end{align*} If $0 \leq d \leq n$, we define $\mathcal{C}_d$ to be the set of all products \begin{align*}A_1 \times \cdots \times A_n \subseteq \mathbb{R}^n, \end{align*} where precisely $d$ of the sets $A_1, \ldots, A_n$ are an element of $\mathcal{J}$ and the remaining $n - d$ ones are an element of $\mathcal{K}$. \\ Set $\mathcal{C} := \mathcal{C}_0 \cup \cdots \cup \mathcal{C}_n$. Then for every $x \in \mathbb{R}^n$ there exists a unique $C \in \mathcal{C}$ such that $x \in C$.

\begin{cla}
Let $0 \leq d \leq n$. For every $C \in \mathcal{C}_d$, there exists an open neighbourhood $U(C)$ of $C$ satisfying:
\begin{enumerate}
\item $\mathrm{diam}~U(C) \leq \frac{3}{2}$
\item $U(C) \cap U(D) = \emptyset$ whenever $D \in \mathcal{C}_d \setminus \{C\}$.
\end{enumerate} \end{cla}
\begin{proof}[Proof of claim]
Let $x = (x_1,\ldots,x_n) \in C$. We will show that there exists a number $0 < \varepsilon(x) \leq \frac{1}{2}$ such that the open cube centered at $x$ with radius $\varepsilon(x)$, i.e. the set \begin{align*}W_{\varepsilon(x)}(x) = (x_1 - \varepsilon(x), x_1 + \varepsilon(x)) \times \cdots \times (x_n - \varepsilon(x), x_n + \varepsilon(x)), \end{align*} intersects no other element of $\mathcal{C}_d$. If $d = 0$, choose $\varepsilon(x) := \frac{1}{2}$. If $d > 0$, exactly $d$ of the numbers $x_1,\ldots,x_n$ are not integers. Choose $0 < \varepsilon(x) \leq \frac{1}{2}$ such that for each $1 \leq i \leq n$ that satisfies $x_i \notin \mathbb{Z}$, the interval $(x_i - \varepsilon(x), x_i + \varepsilon(x))$ contains no integer. If $y = (y_1,\ldots,y_n) \in W_{\varepsilon(x)}(x)$, we have $y_i \notin \mathbb{Z}$ whenever $x_i \notin \mathbb{Z}$. Thus, either $y \in C$ or $y \in C'$ for some $C' \in \mathcal{C}_{d'}$ where $d' > d$. In conclusion, $W_{\varepsilon(x)}(x)$ intersects no other element of $\mathcal{C}_d$. \\ Now let $U(C)$ be the union of all $W_{\frac{\varepsilon(x)}{2}}$, where $x \in C$. Then obviously $U(C) \cap U(D) = \emptyset$ whenever $D \in \mathcal{C}_d \setminus \{C\}$. This proves (2). \\ If $x,y \in U(C)$, we have $x \in W_{\frac{\varepsilon(x')}{2}}(x')$ and $y \in W_{\frac{\varepsilon(y')}{2}}(y')$ for some $x',y' \in C$. By the triangle inequality \begin{align*}\lVert x - y \rVert_{\infty} \leq \lVert x - x' \rVert_\infty + \lVert x' - y' \rVert_\infty + \lVert y' - y \rVert_\infty \leq \frac{1}{4} + 1 + \frac{1}{4} = \frac{3}{2}, \end{align*} hence establishing (1).
\end{proof}
Now let $\mathcal{A} := \{U(C) \; | \; C \in \mathcal{C} \}$. Then $\mathcal{A}$ is an open cover of $\mathbb{R}^n$ of order $n$ by (2). Let $K \subseteq \mathbb{R}^n$ be compact and let $\mathcal{U}$ be an open cover of $K$. Since $K$ is compact metric, $\mathcal{U}$ has a positive Lebesgue number $\lambda > 0$. \\ Consider the homeomorphism $f\colon \mathbb{R}^n \to \mathbb{R}^n, \; x \mapsto \frac{\lambda}{3} \cdot x$. Since $\mathcal{A}$ is an open cover of order $n$, so is $\mathcal{A}' := \{f(U(C)) \; | \; C \in \mathcal{C} \}$. Since $\mathrm{diam}~f(U(C)) \leq \frac{\lambda}{2} < \lambda$ for all $C \in \mathcal{C}$, we get that $\{f(U(C)) \cap K\}_{C \in \mathcal{C}}$ is an open cover of $K$ that refines $\mathcal{U}$ and has order at most $n$. \\ Thus, $\mathrm{dim}~K \leq n$, as desired.
\end{proof}
We need some more elementary properties of the topological dimension before we can proceed to manifolds.
\begin{lem}Let $X$ be a finite-dimensional topological space and let $Y$ be a closed subspace of $X$. Then $Y$ is also finite-dimensional and $\mathrm{dim}~Y \leq \mathrm{dim}~X$. \end{lem}
\begin{proof}
Let $d := \mathrm{dim}~X$. Let $\mathcal{U}$ be an open cover of $Y$. For every $U \in \mathcal{U}$ there exists some open $U' \subseteq X$ such that $U = U' \cap Y$. Let $\mathcal{A} := \{U'\}_{U \in \mathcal{U}} \cup \{X \setminus Y\}$. Then $\mathcal{A}$ is an open cover of $X$ and thus possesses a refinement $\mathcal{B}$ of order at most $d$. Therefore, $\mathcal{V} := \{B \cap Y\}_{B \in \mathcal{B}}$ is an open cover of $Y$ of order at most $d$ that refines $\mathcal{U}$. This proves $\mathrm{dim}~Y \leq d$.
\end{proof}

\begin{prop}
Let $X$ be a topological space and assume $X = X_1 \cup X_2$ for some closed finite-dimensional subspaces $X_1,X_2 \subseteq X$. Then $X$ is also finite-dimensional and \begin{align*}\mathrm{dim}\;X = \max \{\mathrm{dim}\;X_1,\mathrm{dim}\;X_2\}. \end{align*} \end{prop} 
Let us fix a notion for the proof of this theorem. If $\mathcal{U}$ is an open cover of $X$ and $Y \subseteq X$ is a subspace of $X$, we say that $\mathcal{U}$ has \emph{order} $m \in \mathbb{N}_0$ \emph{in} $Y$ if there exists some point $y \in Y$ that is contained in $m+1$ distinct elements of $\mathcal{U}$ and no point of $Y$ is contained in more than $m+1$ distinct elements of $\mathcal{U}$.
\begin{proof}
By Lemma 2.5, it suffices to prove $\mathrm{dim}~X \leq \max\{\mathrm{dim}~X_1,\mathrm{dim}~X_2\}$. 
\begin{cla}
Let $\mathcal{U}$ be an open cover of $X$ and let $Y$ be a closed subspace of $X$ such that $\mathrm{dim}~Y \leq d < \infty$. Then $\mathcal{U}$ possesses a refinement that has order at most $d$ in $Y$. 
\end{cla}
\begin{proof}[Proof of claim]
Let $\mathcal{A} := \{U \cap Y \}_{U \in \mathcal{U}}$. Since $\mathcal{A}$ is an open cover of $Y$ and $\mathrm{dim}~Y \leq d$, there exists a refinement $\mathcal{B}$ of $\mathcal{A}$ of order at most $d$. For every $B \in \mathcal{B}$, there exists some open set $U_B \subseteq X$ such that $B = U_B \cap Y$. Furthermore, there exists some $A_B \in \mathcal{U}$ such that $B \subseteq A_B \cap Y$. Then, $\{U_B \cap A_B\}_{B \in \mathcal{B}} \cup \{U \setminus Y\}_{U \in \mathcal{U}}$ is an open cover of $X$ that refines $\mathcal{U}$ and has order at most $d$ in $Y$.
\end{proof}
Now, let $d := \max \{\mathrm{dim}~X_1,\mathrm{dim}~X_2\}$ and let $\mathcal{U}$ be an open cover of $X$. We need to show that $\mathcal{U}$ has a refinement $\mathcal{V}$ of order at most $d$. \\
Let $\mathcal{A}_1$ be a refinement of $\mathcal{U}$ of order at most $d$ in $X_1$ and let $\mathcal{A}_2$ be a refinement of $\mathcal{A}_1$ of order at most $d$ in $X_2$. We can define a map $f\colon \mathcal{A}_2 \to \mathcal{A}_1$ as follows. For every $U \in \mathcal{A}_2$ choose an element $f(U) \in \mathcal{A}_1$ such that $U \subseteq f(U)$. \\
For all $S \in \mathcal{A}_1$, let $V(S)$ be the union of all $U \in \mathcal{A}_2$ that satisfy $f(U) = S$ and finally let $\mathcal{V} := \{V(S) \}_{S \in \mathcal{A}_1}$. Then, $\mathcal{V}$ is an open cover of $X$: For if $x \in X$, then $x \in U$ for some $U \in \mathcal{A}_2$ and because $U \subseteq V(f(U))$, we can deduce $x \in V(f(U))$. Furthermore, $\mathcal{V}$ refines $\mathcal{A}_1$, because $V(S) \subseteq S$ for every $S \in \mathcal{A}_1$. Since $\mathcal{A}_1$ refines $\mathcal{U}$, the cover $\mathcal{V}$ must refine $\mathcal{U}$. \\ Finally, we need to show that $\mathcal{V}$ has order at most $d$.  Suppose $x \in V(S_1) \cap \cdots \cap V(S_k)$, where the sets $V(S_1),\ldots,V(S_k)$ are distinct. Thus, the sets $S_1,\ldots,S_k$ are distinct. For all $1 \leq i \leq k$, we can find a set $U_i \in \mathcal{A}_2$ such that $x \in U_i$ and $f(U_i) = S_i$, because $x \in V(S_i)$. Because $S_1,\ldots,S_k$ are distinct, so are $U_1,\ldots,U_k$. Thus, we have the following situation: 
\begin{align*}x \in U_1 \cap \cdots \cap U_k \subseteq V(S_1) \cap \cdots \cap V(S_k) \subseteq S_1 \cap \cdots \cap S_k \end{align*}
Because $X = X_1 \cup X_2$, we have $x \in X_1$ or $x \in X_2$. If $x \in X_1$, then $k \leq d+1$, because $\mathcal{A}_1$ has order at most $d$ in $X_1$. If $x \in X_2$, we can also conclude $k \leq d+1$, because $\mathcal{A}_2$ has order at most $d$ in $X_2$. \\ Thus, $k \leq d+1$, proving that $\mathcal{V}$ has order at most $d$, as desired.
\end{proof}
A simple induction argument then yields the following corollary.
\begin{coro}
Let $X$ be a topological space and let $X_1,\ldots,X_r \subseteq X$ be closed finite-dimensional subspaces of $X$ such that 
\begin{align*}X = \bigcup_{i = 1}^r X_i. \end{align*}
Then $X$ is also finite-dimensional and 
\begin{align*}\mathrm{dim}~X = \max\{\mathrm{dim}~X_1,\ldots,\mathrm{dim}~X_n\}. \end{align*} 
\end{coro}
We can now apply these results to manifolds.
\begin{coro}
Let $M$ be a topological $n$-manifold. If $C \subseteq M$ is compact, then $\mathrm{dim}~C \leq n$.
\end{coro}
\begin{proof}
Since $M$ is locally Euclidean, $C$ can be covered by finitely many compact $n$-balls $B_1,\ldots,B_k \subseteq M$. By Theorem 2.4 and Lemma 2.5 \begin{align*}\mathrm{dim}~(B_j \cap C) \leq \mathrm{dim}~B_j \leq n \end{align*} (note that $B_j$ is homeomorphic to a compact subset of $\mathbb{R}^n$) for all $1 \leq j \leq k$. \\
Since $C = \bigcup_{j = 1}^k (B_j \cap C)$, Corollary 2.7 yields $\mathrm{dim}~C \leq n$.
\end{proof}

As a special case, we can note that every compact $n$-manifold is finite-dimensional and its topological dimension is at most $n$. In fact, this result can be extended to general $n$-manifolds. For this, we need a technical lemma.
\begin{lem}
Let $X$ be a topological space and assume $X = \bigcup_{i=0}^\infty C_i$, where every $C_i$ is closed, $C_0 = \emptyset$, $C_i \subseteq \topint{C_{i+1}}$ and there exists some $d \in \mathbb{N}_0$ such that $\mathrm{dim}~\overline{C_{i+1} \setminus C_i} \leq d$ for all $i \in \mathbb{N}_0$. \\ Then $X$ is finite-dimensional and $\mathrm{dim}~X \leq d$.
\end{lem}
\begin{proof}
We will construct a sequence of covers $(\mathcal{V}_i)_{i \in \mathbb{N}_0}$ of $X$ such that 
\begin{enumerate} \item $\mathcal{V}_{i+1}$ refines $\mathcal{V}_{i}$, \item $\mathcal{V}_i$ has order at most $d$ in $C_i$, \item if $V \in \mathcal{V}_i$ such that $V \cap C_{i-1} \neq \emptyset$, then $V \in \mathcal{V}_{i+1}$ and \item $\mathcal{V}_0 := \mathcal{U}$. \end{enumerate} Under these hypotheses, 
\begin{align*}\mathcal{V} := \{V \subseteq X \; | \; \exists i \in \mathbb{N}\colon V \in \mathcal{V}_i \; \mathrm{and} \; V \cap C_{i-1} \neq \emptyset \} \end{align*} is a refinement of $\mathcal{U}$ of order at most $d$: Let $x \in X$. Then $x \in C_{i-1}$ for some $i \in \mathbb{N}$. Since $\mathcal{V}_i$ is an open cover of $X$, we get $x \in V$ for some $V \in \mathcal{V}_i$. But this means $V \cap C_{i-1} \neq \emptyset$ and hence $V \in \mathcal{V}$, proving that $\mathcal{V}$ is an open cover of $X$. Suppose now that $U_1,\ldots,U_k$ are distinct elements of $\mathcal{V}$ having nonempty intersection and let $x$ be an element of their intersection. Then, there exists some $i_0 \in \mathbb{N}$ such that $x \in C_{i_0 - 1}$. For each $1 \leq j \leq k$, there exists some $i_j \in \mathbb{N}$ such that $U_j \in \mathcal{V}_{i_j}$ and $U_j \cap C_{i_j - 1} \neq \emptyset$. Letting $i := \max\{i_0,i_1,\ldots,i_k\}$, we get $U_1,\ldots,U_k \in \mathcal{V}_i$ by (3) and 
\begin{align*}x \in \bigcap_{j=1}^k U_j \cap C_{i}. \end{align*}
Since $\mathcal{V}_i$ has order at most $d$ in $C_{i}$ by (2), we get $k \leq d+1$, i.e. $\mathcal{V}$ has order at most $d$,  \vspace*{0.25cm}  as desired.

All that is left now is constructing the sequence $(\mathcal{V_i})_{i \in \mathbb{N}}$. Set $\mathcal{V}_0 = \mathcal{U}$ and suppose $\mathcal{V}_1,\ldots,\mathcal{V}_i$ have already been constructed. Just as in the proof of Theorem 2.6 we can find a refinement $\mathcal{W}$ of $\mathcal{V}_n$ that has order at most $d$ in $\overline{C_{i+1} \setminus C_i}$. Define a map $f\colon \mathcal{W} \to \mathcal{V}_i$ by choosing $f(W)$ such that $W \subseteq f(W)$ for all $W \in \mathcal{W}$. For $U \in \mathcal{V}_i$, we define $V(U)$ to be the union of all $W \in \mathcal{W}$ such that $f(W) = U$. We define $\mathcal{V}_{i+1}$ to consist of three types of set: $\mathcal{V}_{i+1}$ contains all $U \in \mathcal{V}_i$ such that $U \cap C_{i-1} \neq \emptyset$. Furthermore, $\mathcal{V}_{i+1}$ contains all $V(U)$ where $U \in \mathcal{V}_i$ such that $U \cap C_{i-1} = \emptyset$ and $U \cap C_i \neq \emptyset$. Finally, $\mathcal{V}_{i+1}$ contains all $W \in \mathcal{W}$ such that $W \cap C_i \neq \emptyset$. 
\begin{cla}
$\mathcal{V}_{i+1}$ is a refinement of $\mathcal{V}_i$ that has order at most $d$ in $C_{i+1}$. 
\end{cla}
\begin{proof}
Let $x \in X$. We need to show the existence of some $U \in \mathcal{V}_{i+1}$ satisfying $x \in U$. \\ Suppose $x \in \mathcal{C}_{i-1}$. Since $\mathcal{V}_i$ is an open cover of $X$, we have $x \in U$ for some $U \in \mathcal{V}_i$. Because of $U \cap C_{i-1} \neq \emptyset$, we can conclude $U \in \mathcal{V}_{i+1}$. If $x \notin C_{n-1}$, we can find $W \in \mathcal{W}$ satisfying $x \in W$. If $W \cap C_i = \emptyset$, then $W \in \mathcal{V}_{i+1}$. Otherwise, $f(W) \subseteq W$. If $f(W) \cap C_{i-1} \neq \emptyset$, then $x \in f(W) \in \mathcal{V}_{n+1}$. If $f(W) \cap C_{i-1} = \emptyset$, then $x \in V(f(W)))$ and $V(f(W)) \in \mathcal{V}_{i+1}$, because $f(W) \cap C_{i - 1} = \emptyset$ and $\emptyset \neq W \cap C_i \subseteq f(W) \cap C_i$. \\
In conclusion, $\mathcal{V}_{i+1}$ is an open cover of $X$. It is obvious that $\mathcal{V}_{i+1}$ refines $\mathcal{V}_i$. \\ Now let $U_1,\ldots,U_k \in \mathcal{V}_{i+1}$ be $k$ distinct subsets of $\mathcal{V}_{i+1}$ and suppose $x \in C_{i+1}$ such that $x \in \bigcap_{j=1}^k U_j$. If $x \in C_{i-1}$, then necessarily $U_1,\ldots,U_k \in \mathcal{V}_i$ by the definition of $\mathcal{V}_{i+1}$ and thus $k \leq d+1$, because $\mathcal{V}_i$ has order at most $d$ in $C_i$. \\ If $x \in C_i \setminus C_{i-1}$, then $U_1 = V(S_1), \ldots, U_k = V(S_k)$ for some distinct $S_1,\ldots,S_k \in \mathcal{V}_i$ satisfying $S_j \cap C_{i-1} = \emptyset$ and $S \cap C_i \neq \emptyset$ ($1 \leq j \leq k$). Thus, 
\begin{align*}x \in \bigcap_{j =  1}^k V(S_j) \subseteq \bigcap_{j = 1}^k S_j, \end{align*}
implying that $k \leq d+1$, because $\mathcal{V}_i$ has order at most $d$ in $C_i$. \\
Finally if $x \in C_{i+1} \setminus C_i$, then $U_1,\ldots,U_k \in \mathcal{W}$, hence  $k \leq d+1$, because $\mathcal{W}$  has order at most $d$ in $\overline{C_{i+1} \setminus C_i}$. In conclusion, $\mathcal{V}_{i+1}$ has order at most $d$ in $C_{i+1}$.
\end{proof}
This completes the proof of Lemma 2.9.
\end{proof}
If $X$ is a second-countable locally compact Hausdorff space, then we can decompose $X$ as in the statement of Lemma 2.9.
\begin{lem}
Every second-countable locally compact Hausdorff space $X$ can be \emph{exhausted by compact subsets}, i.e. there exist compact subsets $(C_i)_{i \in \mathbb{N}}$ such that $C_i \subseteq \topint{C_{i+1}}$ and $X = \bigcup_{i=1}^\infty C_i$. 
\end{lem}
\begin{proof}
Let $\mathcal{B}$ be a countable basis of the topology of $X$ and let 
\begin{align*}\mathcal{B}' := \{V \in \mathcal{B} \; | \; \overline{V} \; \mathrm{is \; compact} \}. \end{align*}
Since $X$ is locally compact, $\mathcal{B}'$ is again a basis of $X$. Let us now write $\mathcal{B}' = \{V_i\}_{i \in \mathbb{N}}$. Let $C_1 := \overline{V_1}$. Assume now, that compact subsets $C_1,\ldots,C_k$ satisfying $V_j \subseteq C_j$ and $C_{j - 1} \subseteq \topint{C_j}$ for all $1 \leq j \leq k$ (where $C_0 := \emptyset$) have already been constructed. Because $C_k$ is compact, there exists some $m_k \leq k+1$ satisfying $C_k \subseteq \bigcup_{j = 1}^{m_k} V_j$. Letting $C_{k+1} := \bigcup_{j = 1}^{m_k} \overline{V_j}$, we see that $C_{k+1}$ is compact and $C_k \subseteq \topint{C_{k+1}}$ as well as $V_{k+1} \subseteq C_{k+1}$. Thus $(C_i)_{i \in \mathbb{N}}$ is an exhaustion of $X$ by compact subsets.
\end{proof}
Now, we can finally prove that all topological manifolds are finite-dimensional.
\begin{thm}
Let $M$ be a topological $n$-manifold. Then $M$ is finite-dimensional and $\mathrm{dim}~M \leq~n$. 
\end{thm}
\begin{proof}
Since $M$ is a second-countable locally compact Hausdorff space, $M$ can be exhausted by compact subsets $(C_i)_{i \in \mathbb{N}}$. Each $C_i$ is closed and furthermore each $\overline{C_{i+1} \setminus C_i}$ is compact since $\overline{C_{i+1} \setminus C_i} \subseteq C_{i+1}$. Thus, $\mathrm{dim}~\overline{C_{i+1} \setminus C_i} \leq n$ by Corollary 2.8. Lemma 2.9 now yields $\mathrm{dim}~M \leq n$.
\end{proof}

\section{The existence of a locally flat closed embedding}
We want to prove that a topological $n$-manifold $M$ admits a locally flat closed embedding $\iota\colon M \to \mathbb{R}^{2n+1}$. As already discussed in the Introduction, we first show that $M$ admits a topological proper embedding $\iota_0\colon M \to \mathbb{R}^{2n+1}$. The main ingredient is the following theorem. 
\begin{thm}
Let $X$ be a second-countable locally compact Hausdorff space such that every compact subspace of $X$ has dimension at most $n \in \mathbb{N}$. Then  $X$ admits an embedding $\iota\colon X \hookrightarrow \mathbb{R}^{2n+1}$ that is \emph{proper}, i.e. ${\iota}^{-1}(K) \subseteq X$ is compact whenever $K \in \mathbb{R}^{2n+1}$ is compact.
\end{thm}
\begin{coro} 
Every $n$-manifold $M$ admits a locally flat closed embedding $\iota\colon M \hookrightarrow \mathbb{R}^{2n+1}$. 
\end{coro} 
\begin{proof} 
By Corollary 2.8, $M$ satisfies the conditions of Theorem 3.1, thus $M$ admits a topological proper embedding $\iota_0\colon M \hookrightarrow \mathbb{R}^{2n+1}$. If $n \geq 2$, Theorem 1.4 yields a locally flat proper embedding $\iota\colon M \hookrightarrow \mathbb{R}^{2n+1}$. The embedding $\iota$ is also closed since it is a proper map between locally compact Hausdorff spaces (\cite[Chapter I, Proposition 11.5]{Bre97}).
\end{proof}

\vspace*{0.5cm}
If $X$ is a topological space, we denote by $C(X,\mathbb{R}^N)$ the set of all continuous maps $X \to \mathbb{R}^N$. We shall equip $\mathbb{R}^N$ with the metric 
\begin{align*}\delta(x,y) := \min\{1,\lVert x - y \rVert_\infty\}, \end{align*} where $x,y \in \mathbb{R}^N$.
Then $\delta$ induces the same topology on $\mathbb{R}^N$ as $\lVert \cdot \rVert_\infty$ and $(\mathbb{R}^N, \delta)$ is a complete metric space. We equip $C(X,\mathbb{R}^N)$ with the metric 
\begin{align*}\rho(f,g) := \sup_{x \in X} \delta(f(x),g(x)),\end{align*} where $f,g \in C(X,\mathbb{R}^N)$. Since $(R^N,\delta)$ is complete, so is $(C(X,\mathbb{R}^N),\rho)$.

\vspace*{0.5cm}

\begin{define}
Let $X$ be a topological space and let $f \in C(X,\mathbb{R}^N)$. We write $f(x) \xrightarrow{x \to \infty} \infty$, if for all $R > 0$ there exists some compact subset $C \subseteq X$ such that $\lVert f(x) \rVert_\infty > R$ for all $x \in X \setminus C$.
\end{define}
\begin{rem}
Note that $f(x) \xrightarrow{x \to \infty} \infty$ whenever $X$ is compact.
\end{rem}

\begin{lem}
Let $X$ be a topological space and let $f,g \in C(X,\mathbb{R}^N)$ such that $\rho(f,g) < 1$ and $f(x) \xrightarrow{x \to \infty} \infty$. Then also $g(x) \xrightarrow{x \to \infty} \infty$.
\end{lem}
\begin{proof}
Let $R > 0$. There exists some compact subset $C \subseteq X$ such that $\lVert f(x) \rVert_\infty > R + 1$ whenever $x \in X \setminus C$. The triangle inequality yields 
\begin{align*}\lVert f(x) \rVert_\infty \leq \lVert g(x) \rVert_\infty + \lVert f(x) - g(x) \rVert_\infty < \lVert g(x) \rVert_\infty + 1 \end{align*} and hence $\lVert g(x) \rVert_\infty > R$ whenever $x \in X \setminus C$. This proves $g(x) \xrightarrow{x \to \infty} \infty$.
\end{proof}

\begin{lem}
Let $f \in C(X,\mathbb{R}^N)$ such that $f(x) \xrightarrow{x \to \infty} \infty$. Then $f$ is \emph{proper}, i.e. $f^{-1}(K)$ is compact whenever $K \subseteq \mathbb{R}^N$ is compact. 
\end{lem}
\begin{proof}
Let $K \subseteq \mathbb{R}^N$ be compact. Thus, $K \subseteq [-R,R]^N$ for some $R > 0$. We can find a compact subset $C \subseteq X$ such that $\lVert f(x) \rVert_\infty > R$ whenever $x \in X \setminus C$. Therefore, $f^{-1}(K) \subseteq f^{-1}\left([-R,R]^N\right) \subseteq C$. This shows that $f^{-1}(K)$ is compact as a closed subset of the compact space $C$. Therefore, $f$ is proper. 
\end{proof}
Suppose $X$ is a second-countable locally compact Hausdorff space. We can choose a metric $d$ on $X$ that induces the topology of $X$ (see \cite[Chapter I, Theorem 12.12]{Bre97}). For every $f \in C(X,\mathbb{R}^N)$ and $C \subseteq X$ compact, we let 
\begin{align*}\Delta(f,C) := \sup_{z \in f(C)} \mathrm{diam}~f^{-1}(\{z\}). \end{align*}

\begin{lem}
Given $\varepsilon > 0$ and $C \subseteq X$ compact, we let 
\begin{gather*}U_{\varepsilon}(C) := \{f \in C(X,\mathbb{R}^N) \; | \; \Delta(f,C) < \varepsilon \}. \end{gather*} 
Then $U_{\varepsilon}(C)$ is open in $C(X,\mathbb{R}^N)$.
\end{lem}
\begin{proof}
Let $f \in U_{\varepsilon}(C)$ and let $b > 0$ such that $\Delta(f,C) < b < \varepsilon$. Furthermore, let 
\begin{align*}A := \{(x,y) \in C \times C \; | \; d(x,y) \geq b \}. \end{align*}
Since $A$ is closed in the compact space $C \times C$, $A$ is also compact. The continuous map 
\begin{align*} X \times X \to \mathbb{R}, (x,y) \mapsto \delta(f(x),f(y)) \end{align*} is strictly positive on $A$ and thus $r := \frac{1}{2} \cdot \min_{(x,y) \in A} \delta(f(x),f(y))$ satisfies $r > 0$. We will show that $B_{\rho}(f,r) \subseteq U_{\varepsilon}(C)$: Let $g \in B_{\rho}(f,r)$, i.e. $\rho(f,g) < r$. If $(x,y) \in A$, then $\delta(f(x),f(y)) \geq 2r$. Since $\delta(f(x),g(x)) < r$ and $\delta(f(y),g(y)) < r$, we get $g(x) \neq g(y)$. Thus, by contraposition, if $g(x) = g(y)$ for some $x,y \in C$, then $(x,y) \notin A$ and thus $d(x,y) < b$. \\ This shows $\Delta(g,C) \leq b < \varepsilon$.
\end{proof}

We recall the notion of \emph{affine independence.}
\begin{define}
A set of points $S \subseteq \mathbb{R}^N$ is \emph{affinely independent} if for all distinct $p_0,\ldots,p_k \in S$ and $\alpha_0,\ldots,\alpha_k \in \mathbb{R}$, the equations \begin{align*}\sum_{i = 0}^k \alpha_i \cdot p_i = 0 \; \; \mathrm{and} \; \; \sum_{i = 0}^k \alpha_i = 0 \end{align*} imply that $\alpha_0 = \cdots = \alpha_k = 0$.
\end{define}
Geometrically speaking, if $S \subseteq \mathbb{R}^N$ is affinely independent and $\mathrm{card}(S) = k$, then the points of $S$ uniquely determine a $k$-plane in $\mathbb{R}^N$. \\ For what follows, we need to recall the notion of \textit{Baire spaces.}
\begin{define}
Let $X$ be a topological space.
\begin{enumerate}
\item A subset of $X$ is \emph{nowhere dense} if its closure has empty interior.
\item $X$ is a \emph{Baire space} if $\bigcup_{i = 1}^\infty A_i$ has empty interior whenever $A_1,A_2,\ldots$ are closed nowhere dense subsets of $X$.
\end{enumerate} 
\end{define}
\begin{thm}[Baire Category Theorem]
If $X$ is either a locally compact Hausdorff space or a complete metric space, then $X$ is a Baire space.
\end{thm} 
For a proof see \cite[Chapter I, Theorem 17.1]{Bre97}.
\begin{lem}
Let $x_1, \ldots, x_n \in \mathbb{R}^N$ be distinct points and let $r > 0$. Then, there exist distinct points $y_1,\ldots,y_n \subseteq \mathbb{R}^N$ such that:
\begin{enumerate}
  \item $\lVert x_i - y_i \rVert_\infty < r$ for all $1 \leq i \leq n$.
  \item $\{y_1,\ldots,y_n\}$ is in \emph{general position}, i.e. every subset $S \subseteq \{y_1,\ldots,y_n\}$ such that \\ ${\mathrm{card}(S) \leq N+1}$ is affinely independent. 
\end{enumerate}
\end{lem}
\begin{proof}
We construct the points $y_1,\ldots,y_n$ inductively. Let $y_1 := x_1$. Now, suppose $y_1,\ldots,y_k$ have already been constructed and are in general position as well as $\lVert x_i - y_i \rVert_\infty < r$ for all $1 \leq i \leq k$. Consider the union $P$ of all the affine subspaces that are generated by subsets $A \subseteq \{y_1,\ldots,y_k\}$ such that $\mathrm{card}(A) \leq N$. Since every $l$-plane in $\mathbb{R}^N$ is closed and has empty interior whenever $l < N$, we can deduce $\topint{P} = \emptyset$, because $\mathbb{R}^N$ is a Baire space. Thus we can find an $y_{k+1} \in \mathbb{R}^N \setminus P$ satisfying $\lVert x_{k+1} - y_{k+1} \rVert_\infty < r$ proving (1). \\ Now consider a subset $S \subseteq \{y_1,\ldots,y_{k+1}\}$ such that $\mathrm{card}(S) \leq N+1$. If $y_{k+1} \not \in S$, then $S$ is affinely independent by induction hypothesis. If $y_{k+1} \in S$, then $S$ is affinely independent since $y_{k+1} \not \in P$. This proves (2). \\ Thus this process yields the sought points $y_1,\ldots,y_n$.
\end{proof}
Another fact from point-set topology that we need are partitions of unity and the Tietze Extension Theorem. We shall only state the results here and omit the proof.
\begin{thm}
Let $X$ be a paracompact space and let $\mathcal{U} = \{U_i\}_{i \in I}$ be an open cover of $X$. \\ Then there exists a \emph{partition of unity} $\{\phi_i\}_{i \in I}$ \emph{subordinate to} $\mathcal{U}$, i.e. 
\begin{enumerate}
  \item Each $\phi_i\colon X \to [0,1]$ is a continuous map.
 \item $\mathrm{supp}~\phi_i \subseteq U_i$ for all $i \in I$.
 \item $\{\mathrm{supp}~\phi_i\}_{i \in I}$ is locally finite, i.e. each point $x \in X$ has a neighbourhood that intersects only finitely many of the $\{\mathrm{supp}~\phi_i\}_{i \in I}$.
 \item $\sum_{i \in I} \phi_i(x) = 1$ for all $x \in X$.
\end{enumerate}
\end{thm}
For a proof see \cite[Theorem 41.7]{Mun00}. Recall that second-countable locally compact Hausdorff spaces are paracompact. 

\begin{thm}[Tietze Extension Theorem]
Let $X$ be a \emph{normal space}, i.e. singletons are closed in $X$ and disjoint closed subsets of $X$ can be separated by disjoint neighbourhoods, and let $A \subseteq X$ be closed and let $a,b \in \mathbb{R}, ~ a < b$.
\begin{enumerate}
\item Every continuous map $f\colon A \to [a,b]$ can be extended to a continuous map $F\colon X \to [a,b]$.
\item Every continuous map $f\colon A \to \mathbb{R}$ can be extended to a continuous map $F\colon X \to \mathbb{R}$.
\end{enumerate}
\end{thm}
For a proof see \cite[Theorem 35.1]{Mun00}.
\vspace*{0.5cm}
\begin{lem}
Suppose, $X$ is a second-countable locally compact Hausdorff space such that every compact subspace of $X$ has topological dimension at most $n \in \mathbb{N}$. If $\emptyset \neq C \subseteq X$ is compact, then $U_{\varepsilon}(C)$ is dense in $C(X,\mathbb{R}^N)$ for every $\varepsilon > 0$.
\end{lem}
\begin{proof}
Choose a metric $d$ on $X$ and let $f \in C(X,\mathbb{R}^{2n+1})$ and let $1 > r > 0$. We need to find a $g \in U_{\varepsilon}(C)$ satisfying $\rho(f,g) \leq r$.  Since $C$ is compact, we can cover $C$ by finitely many open (open in $C$) sets $U_1,\ldots,U_m \subseteq C$ such that 
\begin{enumerate}
  \item $\mathrm{diam}~U_i < \frac{\varepsilon}{2}$ for all $1 \leq i \leq m$,
  \item $\mathrm{diam} f(U_i) \leq \frac{r}{2}$ for all $1 \leq i \leq m$,
  \item $\{U_1,\ldots,U_m\}$ has order at most $n$.
\end{enumerate}
Let $\{\phi_1,\ldots,\phi_m\}$ be a partition of unity subordinate to $\{U_1,\ldots,U_m\}$. For each $1 \leq i \leq m$ choose a point $x_i \in U_i$. Then choose $z_1,\ldots,z_m \in \mathbb{R}^{2n+1}$ such that $\lVert f(x_i) - z_i \rVert_\infty < \frac{r}{2}$ and $\{z_1,\ldots,z_m\}$ is in general position (Lemma 3.10). Finally, let 
\begin{align*}\tilde{g}\colon C \to \mathbb{R}^{2n+1}, \; x \mapsto \sum_{i = 1}^m \phi_i(x) \cdot z_i. \end{align*}
\begin{cla}
$\lVert \tilde{g}(x) - f(x) \rVert_\infty < r$ for all $x \in C$. 
\end{cla}
\begin{proof}[Proof of claim]
For all $x \in C$, we have 
\begin{align*}\tilde{g}(x) - f(x) = \sum_{i = 1}^m \phi_i(x) \cdot (z_i - f(x_i)) + \sum_{i = 1}^m \phi_i(x) \cdot (f(x_i) - f(x)), \end{align*}
where we have used $\sum_{i = 1}^m \phi_i(x) = 1$. We have $\lVert z_i - f(x_i) \rVert_\infty < \frac{r}{2}$ for all $1 \leq i \leq m$. Also if $\phi_i(x) \neq 0$, then $x \in U_i$ and since $\mathrm{diam}~f(U_i) < \frac{r}{2}$ by (2), we can conclude $\lVert f(x_i) - f(x) \rVert_\infty < \frac{r}{2}$. Thus, 
\begin{gather*}\lVert \tilde{g}(x) - f(x) \rVert_\infty < \sum_{i = 1}^m \phi_i(x) \cdot \frac{r}{2} + \sum_{i = 1}^m \phi_i(x) \cdot \frac{r}{2} = r. \qedhere \end{gather*} 

\end{proof}
\begin{cla}
If $x,y \in C$ satisfy $\tilde{g}(x) = \tilde{g}(y)$, then $d(x,y) < \frac{\varepsilon}{2}$. 
\end{cla}
\begin{proof}[Proof of claim]
We will prove that $\tilde{g}(x) = \tilde{g}(y)$ implies $x,y \in U_i$ for some $1 \leq i \leq m$. Since $\mathrm{diam}~U_i < \frac{\varepsilon}{2}$ by (1), the claim follows. \\
$\tilde{g}(x) = \tilde{g}(y)$ implies $\sum_{i=1}^m (\phi_i(x) - \phi_i(y)) \cdot z_i = 0$. Because the cover $\{U_1,\ldots,U_m\}$ has order at most $n$ by (3), at most $n+1$ of the numbers $\phi_1(x),\ldots,\phi_m(x)$ and at most $n+1$ of the numbers $\phi_1(y),\ldots, \phi_m(y)$ are non zero. Letting 
\begin{align*}S := \{z_i \; | \; 1 \leq i \leq m \; \mathrm{and} \; \phi_i(x) - \phi_i(y) \neq 0 \},\end{align*} we can deduce $\mathrm{card}(S) \leq 2n+2$.
Note that $\sum_{i = 1}^m (\phi_i(x) - \phi_i(y)) = 0$ and since  $\{z_1,\ldots,z_m\} \subseteq \mathbb{R}^{2n+1}$ are in general position and $\mathrm{card}(S) \leq 2n+1+1$, we can conclude $\phi_i(x) - \phi_i(y) = 0$ for all $1 \leq i \leq m$. Since $\phi_i(x) > 0$ for some $1 \leq i \leq m$, we get $\phi_i(x) = \phi_i(y) > 0$ and thus $x,y \in U_i$. \qedhere
\end{proof}
In conclusion, 
\begin{align*}h\colon C \to [-r,r]^{2n+1}, \; x \mapsto f(x) - \tilde{g}(x)\end{align*}
is a well-defined continuous map. As a locally compact Hausdorff space, $X$ is also normal. Thus, we can apply the Tietze Extension Theorem: $h$ can be extended to a continuous map $H: X \to [-r,r]^{2n+1}$. Letting 
\begin{align*}g\colon X \to \mathbb{R}^{2n+1}, x \mapsto f(x) - H(x), \end{align*}
we have $g|_C = \tilde{g}$ and thus $\Delta(g,C) \leq \frac{\varepsilon}{2} < \varepsilon$ and $\rho(f,g) \leq r$.
\end{proof}
\vspace*{0.5cm}
Let $X$ be as in Theorem 3.1 or Lemma 3.13 and choose a metric $d$ on $X$. Since $(C(X,\mathbb{R}^{2n+1}),\rho)$ is a Baire space, every intersection of countably many open dense subsets of $C(X,\mathbb{R}^{2n+1})$ is again dense in $C(X,\mathbb{R}^{2n+1})$. Consider an exhaustion of $X$ by compact subsets $(C_k)_{k \in \mathbb{N}}$ (Lemma 2.10). Then the set $\bigcap_{k = 1}^\infty U_{1/k}(C_k)$ is dense in $C(X,\mathbb{R}^{2n+1})$.
\begin{lem}
Every $f \in \bigcap_{k = 1}^\infty U_{1/k}(C_k)$ is injective. 
\end{lem}
\begin{proof}
Let $x,y \in X$ such that $f(x) = f(y)$. There exists some $k_0 \in \mathbb{N}$ such that $x,y \in C_k$ whenever $k \geq k_0$. Because $f \in U_{1/k}(C_k)$, we get $d(x,y) \leq \frac{1}{k}$ for all $k \geq k_0$. Hence, $d(x,y) = 0$ and therefore $x = y$.
\end{proof}
\begin{lem}
If $X$ is a second-countable locally compact Hausdorff space, then there exists a map $f \in C(X,\mathbb{R}^N)$ such that $f(x) \xrightarrow{x \to \infty} \infty$. 
\end{lem}
\begin{proof}
It suffices to consider the case $N = 1$. Let $\{U_k\}_{k \in \mathbb{N}}$ be cover of $X$ by open sets such that $\overline{U_k}$ is compact for each $k \in \mathbb{N}$. Since $X$ is second-countable locally compact Hausdorff, $X$ is paracompact and we can find a partition of unity $\{\phi_k\}_{k \in \mathbb{N}}$ subordinate to $\{U_k\}_{k \in \mathbb{N}}$. Letting 
\begin{align*}f\colon X \to \mathbb{R},\; x \mapsto \sum_{k = 1}^\infty k \cdot \phi_k(x), \end{align*} we see that $f(x) \xrightarrow{x \to \infty} \infty$.
\end{proof}
\vspace*{0.5cm}
We can now proceed to the proof of Theorem 3.1.
\begin{proof}[Proof of Theorem 3.1]
Begin with a continuous map $f\colon X \to \mathbb{R}^{2n+1}$ such that $f(x) \xrightarrow{x \to \infty} \infty$ from Lemma 3.15. Consider an exhaustion of $X$ by compact subsets $(C_k)_{k \in \mathbb{N}}$ (Lemma 2.10). Since $\bigcap_{k = 1}^\infty U_{1/k}(C_k)$ is dense in $C(X,\mathbb{R}^{2n+1})$, we can find $\iota \in \bigcap_{k = 1}^\infty U_{1/k}(C_k)$ such that $\rho(f,\iota) < 1$. Then $\iota$ is injective by Lemma 3.13 and $\iota(x) \xrightarrow{x \to \infty} \infty$ by Lemma 3.4. Then, $\iota\colon M \hookrightarrow \mathbb{R}^{2n+1}$ is a proper embedding by Lemma 3.5, as desired.
\end{proof}

\section{ANRs, ENRs and normal microbundles}
As already stated in the introduction, every embedded smooth $n$-manifold $M \subseteq \mathbb{R}^{2n+1}$ admits a \emph{tubular neighbourhood} $U \subseteq \mathbb{R}^{2n+1}$ with a retraction $r\colon U \to M$. In this last section, we want to explain how one proves the analogous results for topological manifolds. 

\begin{define}
A topological space $X$ is a \emph{Euclidean Neighbourhood Retract (ENR)} if there exists a closed embedding $\iota \colon X \hookrightarrow \mathbb{R}^N$ for some $N \in \mathbb{N}$ and an open neighbourhood $U \subseteq \mathbb{R}^N$ of $\iota(X)$ such that $\iota(X)$ is a retraction of $U$, i.e. there exists a continuous map $r \colon U \to \iota(X)$ satisfying $r|_{\iota(X)} = \mathrm{id}_{\iota(X)}$. 
\end{define}

\begin{define}
A topological space $X$ is called an \emph{Absolute Neighbourhood Retract (ANR)} if for every paracompact space $P$ and every continuous map $f\colon A \to X$, where $A \subseteq P$ is closed, there exists an extension $\overline{f}\colon W \to X$ of $f$ where $W$ is an open neighbourhood of $A$.
\end{define}
Why are we interested in ANRs? Here is the answer.
\begin{lem}
If $X$ is an ANR that admits a closed topological embedding $\iota\colon X \hookrightarrow \mathbb{R}^N$ for some $N \geq 0$, then $X$ is an ENR. 
\end{lem}
\begin{proof}
Since $X$ is an ANR, so is $\iota(X) \subseteq \mathbb{R}^N$. Since $\mathbb{R}^N$ is paracompact and $\iota(X) \subseteq \mathbb{R}^N$ is closed, we can extend $f\colon \iota(X) \to \iota(X), \; x \mapsto x$ to a continuous map $r\colon U \to \iota(X)$ where $U$ is some neighbourhood of $\iota(X)$. This $r$ is a retraction.
\end{proof}
Thus, we are finished, if we can show that every topological manifold is an ANR. \\
The following theorem, which is the main ingredient for showing that every topological manifold is an ANR, is due to Hanner \cite{Han51}.
\begin{thm}[{\cite[Theorem 3.3]{Han51}}]
If a topological space $X$ is the union of open ANRs, then $X$ is an ANR.
\end{thm}
\begin{lem}
Every open subset of an ANR is again an ANR.
\end{lem}
\begin{proof}
Let $X$ be an ANR and let $U \subseteq X$ be open. Let $f\colon A \to U$ be continuous where $A \subseteq P$ is closed, $P$ is paracompact. Letting $\tilde{f} := i \circ f$, where $i\colon U \hookrightarrow X$ is the standard embedding, $\tilde{f}$ can be extended to a map $\overline{f}\colon W \to X$, where $W$ is an open neighbourhood of $A$. Then, $\overline{f}|_{\overline{f}^{-1}(U)}$ is the sought extension.
\end{proof} 
\begin{coro} \leavevmode 
\begin{enumerate} 
\item Every topological $n$-manifold $M$ is an ANR.
\item Every topological $n$-manifold $M$ is an ENR. \end{enumerate}
\end{coro}
\begin{proof}
By the Tietze Extension Theorem, the closed unit interval $I = [0,1]$ is an ANR. Thus, $I^n$ is also an ANR. Since every point $p \in M$ possesses a neighbourhood that is homeomorphic to an open subset of $I^n$, we conclude that $M$ is the union of open ANRs and thus an ANR itself by Theorem 4.4. Since $M$ admits a closed embedding $\iota\colon M \to \mathbb{R}^{2n+1}$, $M$ is also an ENR.
\end{proof}
That every topological manifold $M$ is an ENR was used by Milnor in \cite{Mil64} to investigate under which conditions $M \times \mathbb{R}^q$ admits a smooth structure for sufficiently large $q \geq 0$. This was a precursor to the smoothing and PL-ing theory of Kirby and Siebenmann \cite{KS77}. A key ingredient are \textit{microbundles.}

\begin{define}[{\cite[§2]{Mil64}}]A \emph{microbundle} $\mathfrak{X}$ is a diagram \newline
\begin{center}
\begin{tikzcd}
	B &&& E &&& B
	\arrow["i", from=1-1, to=1-4]
	\arrow["j", from=1-4, to=1-7]
\end{tikzcd}
\end{center} 
consisting of the following:

\begin{enumerate}
\item a topological space $B$ called the \emph{base space},
\item a topological space $E = E\left(\mathfrak{X}\right)$ called the \emph{total space}
\item continuous maps $i$ and $j$ such that $j \circ i = \mathrm{id}_B$. \end{enumerate} 

Furthermore, we require that there exists an $r \geq 0$, such that for every $b \in B$, there exists an open neighbourhood $U$ of $b$ and an open neighbourhood $V$ of $i(b)$, with \begin{align*}i(U) \subseteq V, ~~ & j(V) \subseteq U, \end{align*} such that $V \cong U \times \mathbb{R}^r$ under a homeomorphism, such that the following diagram commutes: \newline

\begin{center}
\begin{tikzcd}
	&& V \\
	U &&&& U \\
	&& {U \times \mathbb{R}^r}
	\arrow["{i|_U}", from=2-1, to=1-3]
	\arrow["{\mathrm{id}_U \times 0}"', from=2-1, to=3-3]
	\arrow["{\cong }", from=1-3, to=3-3]
	\arrow["{\mathrm{pr}_1}"', from=3-3, to=2-5]
	\arrow["{j|_V}", from=1-3, to=2-5]
\end{tikzcd}
\end{center}
We call $r$ the \emph{fibre dimension} of $\mathfrak{X}$.
\end{define}

\begin{define}[{\cite[§5]{Mil64}}]
Let $M$ and $N$ be topological manifolds such that $M \subseteq N$ is a submanifold of $N$. We say that $M$ possesses a \emph{microbundle neighbourhood} in $N$ if there exists a neighbourhood $U$ of $M$ in $N$ and a retraction $r\colon U \to M$, such that the diagram 
\begin{center}
\begin{tikzcd}
	M &&& U &&& M
	\arrow["\mathrm{incl}", from=1-1, to=1-4]
	\arrow["r", from=1-4, to=1-7]
\end{tikzcd}
\end{center} 
constitutes a microbundle $\mathfrak{n}$. We call $\mathfrak{n}$ a \emph{normal microbundle} of $M$ in $N$.
\end{define}
We notice immediately that every embedded smooth manifold $M \subseteq \mathbb{R}^N$ possesses a normal microbundle. Any tubular neighbourhood $U$ is a microbundle neighbourhood of $M$. The following result by Stern yields that every locally flat embedded topological $n$-submanifold $M \subseteq \mathbb{R}^{2n+1}$ admits a normal microbundle.

\begin{thm}[{\cite[Theorem 4.5]{Ste75}}]
Let $N$ be a topological $n$-manifold and let $M$ be a topological $m$-manifold, such that $M \subseteq N$ is a locally flat submanifold of $N$. Let $q = n - m$.
\begin{enumerate}
\item If $m \leq q + 1$ and $q \geq 5$, then there exists a normal microbundle $\mathfrak{n}$.
\item Pick $j \in \{1,2\}$. If $m \leq q + j + 1$ and $q \geq 5 + j$, then there exists a normal microbundle~$\mathfrak{n}$.
\end{enumerate}
\end{thm}

\begin{coro}
Let $M$ be a topological $n$-manifold. Then $M$ admits a locally flat embedding $\iota\colon M \hookrightarrow \mathbb{R}^{2n+1}$, such that $\iota(M)$ possesses a microbundle in $\mathbb{R}^{2n+1}$.
\end{coro}
\begin{proof}
If $n \leq 3$, then $M$ admits a smooth structure (\cite{Hat13} and \cite[Theorem 2]{Ham76} together with \cite[Theorem 2.10]{Mun64}). Thus, $M$ admits a smooth embedding $\iota\colon M \hookrightarrow \mathbb{R}^{2n+1}$. The smooth submanifold $\iota(M)$ admits a normal microbundle by the Tubular Neighbourhood Theorem (\cite[Theorem 6.24 and Proposition 6.25]{Lee13}). If $n \geq 4$, the result follows from Theorem \mbox{4.9}.
\end{proof} 
\bibliographystyle{alpha}
\bibliography{bib}

\end{document}